\documentclass[12pt]{article}
\usepackage{amsmath,amsfonts,amsthm,amssymb,mathrsfs,mathtools,ulem}
\usepackage{graphicx}
\usepackage[usenames,dvipsnames]{color}
\usepackage{float}
\usepackage{graphicx}
\usepackage{subfigure}
\usepackage{caption}
\usepackage{cite}
\usepackage{url}
\usepackage{caption}
\usepackage{epstopdf}
\usepackage{hyperref}

\usepackage{color}

\newcommand{\dint}{\displaystyle\int}

\newcommand\redout{\bgroup\markoverwith
	{\textcolor{red}{\rule[0.5ex]{2pt}{0.8pt}}}\ULon}

\theoremstyle{plain}
\newtheorem{theorem}{Theorem}[section]
\newtheorem{hy}{Assumption}[section]
\newtheorem{corollary}[theorem]{Corollary}
\newtheorem{lemma}[theorem]{Lemma}
\newtheorem{proposition}[theorem]{Proposition}

\theoremstyle{definition}

\theoremstyle{remark}
\newtheorem{remark}[theorem]{Remark}

\numberwithin{equation}{section}
\numberwithin{theorem}{section}

\usepackage{a4wide}
\usepackage{geometry}
\begin{document}
	\renewcommand{\thefootnote}{\fnsymbol{footnote}}
	
	\begin{center}
		{\Large \textbf{McKean-Vlasov Processes of Bridge Type}} \\[0pt]
		~\\[0pt] \textbf{Wolfgang Bock} \footnote[1]{Linnaeus University, Vejdesplats 7, SE-351 95 V\"axj\"o, Sweden. E-mail: \texttt{wolfgang.bock@lnu.se}} \textbf{Astrid Hilbert} \footnote[2]{Linnaeus University, Vejdesplats 7, SE-351 95 V\"axj\"o, Sweden.
			E-mail: \texttt{astrid.hilbert@lnu.se}} and \textbf{Mohammed Louriki} \footnote[3]{Mathematics Department, Faculty of Sciences Semalalia, Cadi Ayyad University, Boulevard Prince Moulay Abdellah, P. O. Box 2390, Marrakesh 40000, Morocco. E-mail: \texttt{m.louriki@uca.ac.ma}}
		\\[0pt]
	\end{center}
	\begin{abstract}
		In this paper, we introduce and study McKean–Vlasov processes of bridge type. Specifically, we examine a stochastic differential equation (SDE) of the form: 
		\begin{equation*}
			\mathrm{d} \xi_t=-\mu(t,\mathbb{E}[\varphi_1(\xi_t)]) \frac{\xi_t}{T-t} \mathrm{d} t+\sigma(t,\mathbb{E}[\varphi_2(\xi_t)]) \mathrm{d} W_t,\,\, t<T,
		\end{equation*} 
		where $\mu$ and $\sigma$ are deterministic functions that depend on time $t$ and the expectation of given functions $\varphi_1$ and $\varphi_2$ of the process, and $W$ is a Brownian motion. We establish the existence and uniqueness of solutions to this equation and analyze the behavior of the process as $t$ approaches $T$. Furthermore, we provide conditions ensuring the pinned property of the process $\xi$. Finally, we explore explicit solutions in specific cases of interest, including power-weighted expectations and second moments in the drift.  
	\end{abstract}
	\section{Introduction}
	
	\hspace{0,6cm} Mean-field stochastic differential equations also McKean–Vlasov stochastic differential equations (MV-SDEs) describe processes, the dynamics of which depend not only on the state of the process but also its distribution. Historically MV-SDEs depended on the mean of the state, which is also the setting in this work. These equations capture interactions between individual and collective behavior, making them valuable tools in various fields, including mean-field theory, interacting particle systems, and financial modeling. Here we refer to the classical papers which devoted to the study of McKean–Vlasov diffusion processes \cite{F}, \cite{Mc1}, \cite{Mc2}, \cite{S}, \cite{Sz}, and, e.g., the more recent papers \cite{BR}, \cite{BR2}, \cite{DST}, \cite{HSS}, \cite{HR}, \cite{HRW}, \cite{HW}, \cite{Gru24} \cite{MV} and \cite{W}. 
	
	The Brownian bridge is a fundamental concept in statistics and probability theory, widely recognized as a powerful tool with diverse applications. For instance, it emerges as the large-population limit of the cumulative sum process when sampling randomly without replacement from a finite population (see \cite{R}). Additionally, it arises in the limit of the normalized difference between a given distribution and its empirical law and is central to the Kolmogorov-Smirnov test. The Brownian bridge also has numerous applications in finance (see, e.g., \cite{B}, \cite{BBE}, \cite{BS}, \cite{K}, \cite{L1}, and \cite{L2}). Since the standard Brownian bridge vanishes at its terminal time $T$, it is considered a suitable candidate for modeling the noise process that represents the flow of information about a non-defaultable cash flow due at time $T$ in the information-based approach introduced by Brody et al. in \cite{BHM2007}. In \cite{M}, the Brownian bridge is generalized by modifying its drift term in the associated SDE, where the drift is scaled by a constant $\alpha$. Specifically, the author studied the process \[ \mathrm{d} Z_t = -\alpha \frac{Z_t}{T-t} \mathrm{d} t + \mathrm{d} B_t, \] where $\alpha \in \mathbb{R}_*^{+}$ and $B$ is a standard Brownian motion. This process is referred to as the $\alpha$-Brownian bridge of length $T$. The study in \cite{M} provides an explicit representation of the $\alpha$-Brownian bridge and investigates its pinned property. Furthermore, it is shown that if $\alpha \neq 1$, the $\alpha$-Brownian bridge of length $T$ is not the bridge of length $T$ for a centered Gaussian Markov process. As an application, the process is used to determine the laws of certain quadratic functionals of Brownian motion. In \cite{BK}, \cite{HR}, and \cite{JY}, the notion of the $\alpha$-Brownian bridge of length $T$ is generalized to continuous functions $\alpha$ defined on the interval $[0,T)$.
	
	Motivated by the previous works, we introduce and study McKean–Vlasov SDEs of
	bridge type. This type of processes has not been considered before. This
	paper explores different types of McKean–Vlasov SDEs, focusing on the
	interplay between power-weighted moments of the solution, transformations of
	the state variable, and time-dependent coefficients in the drift and
	diffusion terms. We first consider McKean–Vlasov SDEs with additive noise
	where the drift incorporates a power-weighted moments of the process itself.
	Specifically, we study equations of the form
	\begin{equation}
		\mathrm{d}X_t = -(\mathbb{E}[X_t^\nu])^\alpha \dfrac{X_t}{T-t} \, \mathrm{d}t
		+ \mathrm{d}W_t, \quad t < T,
		\label{eqintroSDE-MVexp}
	\end{equation}
	with $X_0=x\geq 0$. Here, $W$ represents a Brownian motion, and the parameters
	$\alpha,\nu > 0$ govern the impact of the mean behaviour. In the special case
	$\alpha = 0$, the equation simplifies to the classical SDE for the standard
	Brownian bridge of length $T$.
	For $\alpha,\nu > 0$, however, the drift reflects the mean-field effect. We provide explicit solutions to the SDEv\eqref{eqintroSDE-MVexp} for $\nu=1,\ \alpha > 0$ and $\nu=2,\ \alpha =1> 0$. We furthermore prove that, for $\nu=1,\ \alpha > 0$, these are not the bridge
	of a Gaussian Markov process. Nevertheless, they retain the pinned property.

	While powers of the first moment or mean reflect the trend of the solution
	$X_t$, powers of the second moment capture the squared magnitude. When
	modelling it provides richer information about the spread around the mean,
	variance or intensity of the process and has the interpretation of energy in
	physics and volatility in finance.
	
	We establish existence and uniqueness of a strong solution to the MV-SDE
	\eqref{eqintroSDE-MVexp} and prove that the solution possesses the pinned
	property. In the specific cases where $\nu=1, \ \alpha>0$, or $\alpha=1, \
	\nu =2$ we further show that the MV-SDE \eqref{eqintroSDE-MVexp} admits an
	explicit solution, in particular for $\alpha=1, \ \nu =2, \ x=0$, it is given by:
	\begin{equation*}
		Y_t=\dfrac{\dint_0^t\left(I_1(2 \sqrt{2} \sqrt{T}) K_0(2 \sqrt{2} \sqrt{T - s})
			+ K_1(2 \sqrt{2} \sqrt{T}) I_0(2 \sqrt{2} \sqrt{T - s})\right)^{\frac{1}{2}}\,\mathrm{d}W_s}{\left(I_1(2 \sqrt{2} \sqrt{T}) K_0(2 \sqrt{2} \sqrt{T - t}) + K_1(2 \sqrt{2} \sqrt{T}) I_0(2 \sqrt{2} \sqrt{T - t})\right)^{\frac{1}{2}}}, \quad t <T,
	\end{equation*}
	where,
	$I_n$ is the modified Bessel function of the first kind and $K_n$ is the
	modified Bessel function of the second kind. 
	
	Despite of the relevance of feedback-type drifts in terms of power-weighted
	expectations many real-world systems require dynamics that respond to other
	statistical properties, such as skewness, or non-linear transformations
	\(\varphi_1\) as drift and \(\varphi_2\) as diffusion coefficient.
	Furthermore, the inclusion of explicit time dependence resulting in a drift
	\(\mu(t, \cdot)\) and diffusion coefficent \(\sigma(t, \cdot)\) allows the
	model to capture evolving dynamics over a finite time horizon \([0, T)\),
	which is essential in applications and modelling in financial mathematics,
	physics or biology.
	
	We conclude this work by considering MV-SDEs with non-linear feed-back drift
	of the form:
	\begin{equation}
		\mathrm{d} \xi_t=-\mu(t,\mathbb{E}[\varphi_1(\xi_t)]) \frac{\xi_t}{T-t} \mathrm{d} t
		+\sigma(t,\mathbb{E}[\varphi_2(\xi_t)]) \mathrm{d} W_t,\,\, t<T,
		\label{eqintroMV-SDEgeneral}
	\end{equation}
	with $\xi_0=0$.
	
	We establish existence and uniqueness of solutions to the above MV-SDE under
	appropriate conditions on the deterministic functions $\mu$, $\sigma$,
	$\varphi_1$ and $\varphi_2$. Additionally, we prove that the second moment of
	the process satisfies:
	\begin{equation*}
		\mathbb{E}[\xi_t^2]\mapsto 0 \text{   as  } t\mapsto T.
	\end{equation*}
	Moreover, we show that the solution to the MV-SDE
	\eqref{eqintroMV-SDEgeneral} satisfies the pinned property.
	
	From a modeling perspective, as previously mentioned, these types of
	processes have potential applications across various fields, including
	finance, physics, and biology. Here, we highlight a possible application
	within the information-based approach introduced in \cite{BHM2007} and
	further developed in \cite{BBE}, \cite{EL}, \cite{EHL}, \cite{EHL(Levy)},
	\cite{HHM}, \cite{HHM2015}, \cite{L1}, \cite{L2}, \cite{L3}, \cite{RuYu} and
	references therein. Specifically, the MV-SDEs considered here possess the
	pinned property, meaning that $\xi_t$ almost surely converges to zero as $t$
	approaches $T$. This makes them well-suited for modeling the noise associated
	with the flow of information about a non-defaultable cash flow $H_T$ payable
	at time $T$. The vanishing of $\xi$ at $T$ aligns with the fact that
	investors have perfect information about $H_T$ at maturity, thereby capturing
	the dynamics of information revelation over time. The added value of the
	resulting model lies in the fact that including expectations can potentially
	offer more accurate predictions by integrating a form of rational
	expectation. This means agents (like investors) act based on their
	expectations of future states, which can lead to more realistic modeling of
	market behavior. For instance, in stock price modeling, including the
	expectation can help in capturing the anticipatory actions of market
	participants, leading to better pricing models.
	
	The paper is structured as follows: Section 2 introduces McKean–Vlasov
	stochastic differential equations of bridge type in which the drift depends
	on a power-weighted expectation of the process. Section 3 focuses on the case
	where the drift is driven by a power-weighted second moment of the process.
	In Section 4 we examine the general framework, where both the drift and
	volatility depend on time and the expectation of a general function of the
	process.  Finally, in Appendix A, we present several supporting lemmas for the main results of this paper.
	\section{McKean–Vlasov SDEs with Power-Weighted Expectation in the Drift}
	\hspace{0,6cm} McKean–Vlasov stochastic differential equations (MV-SDEs)
	provide a fundamental framework for modeling systems where the dynamics of an
	individual process depend on collective characteristics such as the mean or
	distribution of the process. These equations naturally arise in various
	fields, including finance, physics, and population dynamics. 	
	
	In this section, we study McKean–Vlasov SDEs of bridge type. Since they allow
	for an explicit solution we start with the case of MV-SDEs where the drift
	depends on the expectation of the process, raised to a parameter $\alpha\geq
	0$. The dynamics are described by the equation:
	\begin{equation}
		\left\{\begin{array}{l}
			\mathrm{d} X_t=-(\mathbb{E}[X_t])^{\alpha} [X_t/(T-t)] \mathrm{d} t+\mathrm{d} W_t,\,\, t<T,\\
			X_0=x>0
		\end{array}\right.\label{eq1}
	\end{equation}	
	where $W$ is a Brownian motion, $\nu=1$,  and $\alpha>0$ is a parameter,
	which modulates the influence of the mean of the process or more technical
	mean-field.
	When $\alpha=0$ the equation reduces to the classical SDE for the standard
	Browian bridge of length $T$ between $x$ and $0$.
	We focus on studying the existence and uniqueness of
	solutions to this equation and investigating their basic properties.
	  Before addressing the existence and uniqueness of this type of process, let us first recall the following:

	\begin{proposition}
		The MV-SDE \eqref{eq1} has an explicit solution given by
		\begin{equation}
			X_t^{(\alpha),T}=\mathbb{E}\Big[X_t^{(\alpha),T}\Big]\left(1+\dint_0^t\dfrac{1}{\mathbb{E}\Big[X_s^{(\alpha),T}\Big]}\mathrm{d}W_s\right), t<T,\label{eqsolutionExp}
		\end{equation}
		where
		\begin{equation}
			\mathbb{E}\Big[X_t^{(\alpha),T}\Big]=\left(\dfrac{1}{a_\alpha-\alpha \log(T-t)}\right)^{\frac{1}{\alpha}},\label{eqExpalpha}
		\end{equation}
		and
		\begin{equation}
			a_{\alpha}=\frac{1}{x^{\alpha}}+\alpha \log(T).
		\end{equation}
	\end{proposition}
	\begin{proof}
		We consider the following SDE:
		\begin{equation}
			\left\{\begin{array}{l}
				\mathrm{d} \bar{X}_t=-\beta(t)[\bar{X}_t/(T-t)] \mathrm{d} t+\mathrm{d} W_t,\,\, t<T\\
				\bar{X}_0=x,
			\end{array}\right.\label{eqSDEbetanew}
		\end{equation}		

		where
		\begin{equation}
			\beta(t)=\dfrac{1}{a_\alpha-\alpha \log(T-t)}, \,\,t<T.
		\end{equation}
		Since  $\beta:[0,T)\longrightarrow \mathbb{R}$ is a continuous function, it follows from Lemma \ref{lmsdebeta} that the SDE \eqref{eqSDEbetanew} has a unique strong solution. Moreover, the solution has the following explicit expression: for all $t \in [0, T)$,
		\begin{multline}
			\bar{X}_t=\exp\left(- \int_{0}^{t}  \frac{1}{(T-u)(a_\alpha-\alpha \log(T-u))} \, \mathrm{d}u\right)\\\left[x+\int_{0}^{t}\exp\left( \int_{0}^{s}  \frac{1}{(T-u)(a_\alpha-\alpha \log(T-u))} \, \mathrm{d}u\right)  \mathrm{d}W_s\right]\\
			=\dfrac{1}{(a_\alpha-\alpha \log(T-t))^{\frac{1}{\alpha}}}\left[1+\int_{0}^{t}(a_\alpha-\alpha \log(T-s))^{\frac{1}{\alpha}}\mathrm{d}W_s\right].
		\end{multline}
		Using the above explicit expression of $\bar{X}$, we see that
		\begin{equation}
			\mathbb{E}[\bar{X}_t]=\dfrac{1}{(a_\alpha-\alpha \log(T-t))^{\frac{1}{\alpha}}},
		\end{equation}
		which implies that $\beta_t=\left(\mathbb{E}[\bar{X}_t]\right)^{\alpha}$. Thus, it follows from \eqref{eqSDEbetanew} that the SDE \eqref{eq1} has a unique strong solution given by
		\begin{equation}
			X_t^{(\alpha),T}=\mathbb{E}\Big[X_t^{(\alpha),T}\Big]\left(1+\dint_0^t\dfrac{1}{\mathbb{E}\Big[X_s^{(\alpha),T}\Big]}\mathrm{d}W_s\right), t<T,
		\end{equation}
		where
		\begin{equation}
			\mathbb{E}\Big[X_t^{(\alpha),T}\Big]=\left(\dfrac{1}{a_\alpha-\alpha \log(T-t)}\right)^{\frac{1}{\alpha}}.
		\end{equation}
		This completes the proof.
	\end{proof}
	\begin{proposition}
		The process $X^{(\alpha),T}$ is a Gaussian process with first moment and covariance function given by:
		\begin{equation}
			\mathbb{E}\Big[X_t^{(\alpha),T}\Big]=\left(\dfrac{1}{a_\alpha-\alpha \log(T-t)}\right)^{\frac{1}{\alpha}}
		\end{equation}
		and
		\begin{multline}
			\text{Cov}(X_s^{(\alpha),T},X_t^{(\alpha),T})=T\,\alpha^{\frac{2}{\alpha}}\exp\Big(\frac{1}{\alpha x^{\alpha}}\Big)\mathbb{E}\Big[X_s^{(\alpha),T}\Big]\mathbb{E}\Big[X_t^{(\alpha),T}\Big]\\
			\times\bigg[\gamma\left(\frac{\alpha+2}{\alpha},\frac{1}{\alpha x^{\alpha}}-\log\left(\frac{T-s}{T}\right)\right)-\gamma\left(\frac{\alpha+2}{\alpha},\frac{1}{\alpha x^{\alpha}}\right)\bigg],
		\end{multline} 
		where,
		$\gamma$ is the lower incomplete gamma function, i.e.,
		\begin{equation}
			\gamma(b,x)=\int_0^{x}u^{b-1}\exp(-u)\mathrm{d}u.
		\end{equation}
	\end{proposition}
	\begin{proof}
		From \eqref{eqsolutionExp} we can see that the process $X^{(\alpha),T}$ is a Gaussian process with first moment given by
		\begin{equation}
			\mathbb{E}\Big[X_t^{(\alpha),T}\Big]=\left(\dfrac{1}{a_\alpha-\alpha \log(T-t)}\right)^{\frac{1}{\alpha}}.
		\end{equation}
		For the covariance function, for $s, t\in \mathbb{R}_+$ such that $s\leq t$, we have
		\begin{align}
			\text{Cov}(X_s^{(\alpha),T},X_t^{(\alpha),T})=\mathbb{E}\Big[X_s^{(\alpha),T}\Big]\mathbb{E}\Big[X_t^{(\alpha),T}\Big]\int_{0}^{s}(a_\alpha-\alpha \log(T-u))^{\frac{2}{\alpha}}\mathrm{d}u.
		\end{align}
		For the integral on the right-hand side of the previous equation, we use the following:
		\begin{multline}
			\int_{0}^{s}(a_\alpha-\alpha \log(T-u))^{\frac{2}{\alpha}}\mathrm{d}u=\alpha^{\frac{2}{\alpha}}\exp\Big(\frac{a_\alpha}{\alpha}\Big)\frac{1}{\alpha}\int_{\frac{1}{x^\alpha}}^{a_\alpha-\alpha \log(T-s)}u^{\frac{2}{\alpha}}\exp\left(-\frac{u}{\alpha}\right)\mathrm{d}u\\
			=T\,\alpha^{\frac{2}{\alpha}}\exp\Big(\frac{1}{\alpha x^{\alpha}}\Big)\int_{\frac{1}{\alpha x^\alpha}}^{\frac{a_\alpha}{\alpha}- \log(T-s)}u^{\frac{\alpha+2}{\alpha}-1}\exp\left(-u\right)\mathrm{d}u\\
			=T\,\alpha^{\frac{2}{\alpha}}\exp\Big(\frac{1}{\alpha x^{\alpha}}\Big)\bigg[\gamma\left(\frac{\alpha+2}{\alpha},\frac{1}{\alpha x^{\alpha}}-\log\left(\frac{T-s}{T}\right)\right)-\gamma\left(\frac{\alpha+2}{\alpha},\frac{1}{\alpha x^{\alpha}}\right)\bigg].
		\end{multline}
		This completes the proof.
	\end{proof}
	From the explicit expression of the process $X^{(\alpha),T}$, we can deduce that for all $t<T$
	\begin{equation}
		\dfrac{X_t^{(\alpha),T}}{\mathbb{E}\left[X_t^{(\alpha),T}\right]}=1+\dint_0^t\dfrac{1}{\mathbb{E}\Big[X_s^{(\alpha),T}\Big]}\mathrm{d}W_s.
	\end{equation}
	Hence, we have the following Corollary:
	\begin{corollary}\label{corindinc}
		The process $M^{(\alpha),T}=(M^{(\alpha),T}_t ,t<T)$ given by
		\begin{equation}
			M^{(\alpha),T}_t=\left(a_\alpha-\alpha \log(T-t)\right)^{\frac{1}{\alpha}}X_t^{(\alpha),T}
		\end{equation}
		is with independent increments.
	\end{corollary}
	\begin{proposition}
		For $\alpha>0$, the process $X^{(\alpha),T}$ is not the bridge of length $T$ of a Gaussian Markov process. That is, there exists no non-degenerate Gaussian Markov process $U$ such as, for any functional $F$ we have
		$$\mathbb{E}[F(U_t,\,t\leq T)\vert U_T=0]=\mathbb{E}[F(X_t^{(\alpha),T},\,t\leq T)].$$
	\end{proposition}
	\begin{proof}
		Let $U$ be a non--degenerate continuous Gaussian Markov process that starts from $x$. Let $U^T$ be the bridge between $x$ and $0$ of length $T$ associated with $U$. We have, for $s, t\in \mathbb{R}_+$ such that $s\leq t$,
		\begin{equation}
			\mathbb{E}[U_t^{T}]=\mathbb{E}[U_t]-\frac{\text{Cov}(U_t,U_{T})}{\text{Var}(U_{T})}\mathbb{E}[U_T]
		\end{equation}
		\begin{align}
			\text{Cov}(U_s^{T},U_t^{T})=\text{Cov}(U_s,U_{t})-\frac{\text{Cov}(U_t,U_{T}) \text{Cov}(U_s,U_{T})}{\text{Var}(U_{T})}.
		\end{align}
		Since $Y$ is Markovian, it follows from \cite[page 86]{RY} that the covariance function takes the following form
		\begin{equation}
			\text{Cov}(U_s,U_{t}) = a(s)a(t)\rho(\inf(s,t))
		\end{equation}
		where $a$ is continuous and does not vanish, and $\rho$ is continuous, strictly positive, and non-decreasing. Hence,
		\begin{equation}
			\mathbb{E}[U_t^{T}]=\mathbb{E}[U_t]-\frac{a(t)\rho(t)}{a(T)\rho(T)}\mathbb{E}[U_T]
		\end{equation}
		\begin{align}
			\text{Cov}(U_s^{T},U_t^{T})=a(s)a(t)\rho(s)\left(1-\frac{\rho(t)}{\rho(T)}\right).
		\end{align}
		This implies that,
		\begin{equation}
			a(t)\rho(t)=c_1(\alpha,x)\mathbb{E}\Big[X_t^{(\alpha),T}\Big]\bigg[\gamma\left(\frac{\alpha+2}{\alpha},\frac{1}{\alpha x^{\alpha}}-\log\left(\frac{T-t}{T}\right)\right)-\gamma\left(\frac{\alpha+2}{\alpha},\frac{1}{\alpha x^{\alpha}}\right)\bigg]
		\end{equation}
		and
		\begin{equation}
			a(t)\left(1-\frac{\rho(t)}{\rho(T)}\right)=c_2(\alpha,x)\mathbb{E}\Big[X_t^{(\alpha),T}\Big]
		\end{equation}
		where $c_1(\alpha,x)$ and $c_2(\alpha,x)$ are strictly positive constants that depend only on $c$ and $x$.
		Thus,
		\begin{equation}
			\dfrac{\rho(t)}{\rho(T)-\rho(t)}=\frac{c_1(\alpha,x)}{c_2(\alpha,x)}\bigg[\gamma\left(\frac{\alpha+2}{\alpha},\frac{1}{\alpha x^{\alpha}}-\log\left(\frac{T-t}{T}\right)\right)-\gamma\left(\frac{\alpha+2}{\alpha},\frac{1}{\alpha x^{\alpha}}\right)\bigg].\label{eqbridgecontradiction}
		\end{equation}
		Since $\rho$ is a continuous, strictly positive, and non-decreasing function, the quantity $\frac{\rho(t)}{\rho(T)-\rho(t)}$ tends to infinity as $t$ approaches $T$. It follows from \eqref{eqbridgecontradiction} that $\Gamma\left(\frac{\alpha+2}{\alpha}\right)$ would tend to infinity, which is a contradiction because $\Gamma\left(\frac{\alpha+2}{\alpha}\right)$ is finite for $\alpha>0$. Here, $\Gamma$ is the complete gamma function.
	\end{proof}
	We now show that even if the process $X^{(\alpha),T}$ is not a bridge of a Gaussian Markov process, $X^{(\alpha),T}$ is still a pinned process. Specifically,
	\begin{equation}
		\mathbb{P}\left(\lim\limits_{t \rightarrow T} X_t^{(\alpha),T}=0\right)=1.
	\end{equation}
	\begin{proposition}
		The process $X^{(\alpha),T}$ is a pinned process.
	\end{proposition}
	\begin{proof}
		From Corollary \ref{corindinc}, we see that for all $t<T$
		\begin{equation}
			X_t^{(\alpha),T}=\dfrac{M^{(\alpha),T}_t}{\left(a_\alpha-\alpha \log(T-t)\right)^{\frac{1}{\alpha}}}\label{eqXalphaMalpha}
		\end{equation}
		where $M^{(\alpha),T}$ is a process with independent increments. The process \( M^{(\alpha),T} \) can be decomposed as follows:
		\begin{equation}
			M^{(\alpha),T}_t = 1 + \sum_{k=1}^{n_t} \left( M^{(\alpha),T}_{t_{k+1}} - M^{(\alpha),T}_{t_k} \right),\label{eqMassum}
		\end{equation}
		where \( t_k \) is a partition of \( [0,t] \) and \( n_t \to \infty \) as \( t \to T \). We have,
		\begin{align*}
			\sum_{k=1}^{\infty} \mathbb{E}\left[\left( M^{(\alpha),T}_{t_{k+1}} - M^{(\alpha),T}_{t_k} \right)^2\right]&\leq 	\sum_{k=1}^{\infty}\displaystyle\int_{t_k}^{t_{k+1}}(a_\alpha-\alpha \log(T-u))^{\frac{2}{\alpha}}\mathrm{d}u\\
			&\leq \displaystyle\int_{0}^{T}(a_\alpha-\alpha \log(T-u))^{\frac{2}{\alpha}}\mathrm{d}u\\
			&\leq T\,\alpha^{\frac{2}{\alpha}}\exp\Big(\frac{1}{\alpha x^{\alpha}}\Big)\Gamma\left(\frac{\alpha+2}{\alpha},\frac{1}{\alpha x^{\alpha}}\right)<+\infty,
		\end{align*}
		where,
		$\Gamma$ is the upper incomplete gamma function, i.e.,
		\begin{equation}
			\Gamma(b,x)=\int_x^{+\infty}u^{b-1}\exp(-u)\mathrm{d}u.
		\end{equation}
		Hence, using the Kolmogorov’s convergence theorem, $\sum\limits_{k=1}^{+\infty} \left( M^{(\alpha),T}_{t_{k+1}} - M^{(\alpha),T}_{t_k} \right)$ converges almost surely. Thus, it follows from \eqref{eqXalphaMalpha} and \eqref{eqMassum} that 
		\begin{equation*}
			\mathbb{P}\left(\lim\limits_{t \rightarrow T} X_t^{(\alpha),T}=0\right)=1.
		\end{equation*}
	
		Which completes the proof. 
	\end{proof} 

	\section{McKean–Vlasov SDEs with Power-Weighted Second Moment in the Drift} 
	\hspace{0,6cm}In the previous section, we examined a class of McKean–Vlasov SDEs where the drift term depends on a power-weighted expectation of the process. This formulation provided a foundation for studying systems driven by the mean behavior of the process. However, in many applications, the second moment of the process—rather than its mean—plays a crucial role in driving the dynamics. This is particularly relevant in settings where the mean is zero, rendering the variance and second moment equivalent. In this section, we focus on McKean–Vlasov SDEs where the drift depends on the second moment of the process, raised to a power parameter $\alpha>0$. This framework captures interactions influenced by the overall magnitude of fluctuations in the system, regardless of the mean behavior. Such models arise in various fields, including physics, population dynamics, and finance, where the size of deviations from equilibrium directly impacts the evolution of the process. More precisely, we study MV-SDEs of the form:
	\begin{equation}
		\left\{\begin{array}{l}
			\mathrm{d} Y_t=-(\mathbb{E}[Y_t^2])^{\alpha}\,[Y_t/(T-t)]\mathrm{d} t+\mathrm{d} W_t,\,\, t<T\\
			Y_0=0,
		\end{array}\right.\label{eqalphavariance}
	\end{equation}		 
	where \(\mathbb{E}[X_t^2]\)  is the second moment of \(X_t\), and \(\alpha > 0\) controls the influence of the second moment on the drift term. This section is structured as follows: first, we analyze the case $\alpha=1$, where explicit solutions can be derived. This special case serves as a benchmark for understanding the general dynamics. Next, we extend our study to the general case $\alpha>0$, providing existence and uniqueness results and exploring the fundamental properties of the solutions. 
	\subsection{The Linear Second Moment Case ($\alpha=1$)}
	\hspace{0,6cm}The case $\alpha=1$ represents the simplest form of second moment-driven McKean–Vlasov SDEs. Here, the second moment enters the drift in a linear manner, allowing for explicit solutions. Such solutions provide a clear understanding of the interaction between the second moment and the system's evolution, serving as a baseline for exploring more complex scenarios. In this subsection, we derive the explicit solution for the MV-SDE: 	
	\begin{equation}
		\left\{\begin{array}{l}
			\mathrm{d} \mathcal{Y}_t=-\mathbb{E}[\mathcal{Y}_t^2] \,[\mathcal{Y}_t/(T-t)] \mathrm{d} t+\mathrm{d} W_t,\,\, t<T\\
			\mathcal{Y}_0=0,
		\end{array}\right.\label{eqvariance}
	\end{equation}
	and examine its properties. Let us first study the existence and the uniqueness of the MV-SDE \eqref{eqvariance}.
	\begin{proposition}
		The MV-SDE \eqref{eqvariance} has a unique strong solution.
	\end{proposition}
	\begin{proof}
		Let $v$ be the solution to the ODE given by:
		\begin{equation}
			\left\{\begin{array}{l}
				v'(t)=-2\,[v^2(t)/(T-t)]+1,\,t<T\\
				v(0)=0.
			\end{array}\right.\label{eqf't2}
		\end{equation}		
		We consider the following SDE:
		\begin{equation}
			\left\{\begin{array}{l}
				\mathrm{d} \bar{\mathcal{Y}}_t=-v(t)\,[\bar{\mathcal{Y}}_t/(T-t)] \mathrm{d} t+\mathrm{d} W_t,\,\, t<T\\
				Y_0=0.
			\end{array}\right.\label{eqvariancef}
		\end{equation}
		It follows from Lemma \ref{lmsdebeta} that the SDE \eqref{eqvariancef} has a unique strong solution. Moreover, for all $t<T$, we have
		\begin{equation}
			\mathbb{E}[\bar{\mathcal{Y}}_t^2]=\int_{0}^{t}\exp\left( -2\int_{s}^{t}  \frac{v(u)}{T-u} \, \mathrm{d}u\right)  \mathrm{d}s.\label{eqsecondmomentofbarY}
		\end{equation}
		Using It\^o's formula, we obtain
		\begin{equation}
			\bar{\mathcal{Y}}_t^2=-2\int_{0}^{t}v(s)\dfrac{\bar{\mathcal{Y}}^2_s}{T-s}\mathrm{d}s+2\int_{0}^{t}\bar{\mathcal{Y}}_s\mathrm{d}W_s+t.\label{eqitoY}
		\end{equation}
		It follows from \eqref{eqsecondmomentofbarY} that the process $(\int_{0}^{t}\bar{\mathcal{Y}}_s\mathrm{d}W_s, 0\leq t<T)$ is a (true) martingale. Consequently, it follows from \eqref{eqitoY} that
		\begin{equation}
			\mathrm{d}\mathbb{E}[\bar{\mathcal{Y}}_t^2]=-2v(t)\dfrac{\mathbb{E}[\bar{\mathcal{Y}}^2_t]}{T-t}\mathrm{d}t+\mathrm{d}t.\label{eqEY_t^2'}
		\end{equation}
		Using \eqref{eqf't2} and \eqref{eqEY_t^2'} we get
		\begin{align*}
			\mathrm{d}(v(t)-\mathbb{E}[\bar{\mathcal{Y}}_t^2])&=-2\dfrac{v(t)}{T-t}(v(t)-\mathbb{E}[\bar{\mathcal{Y}}_t^2])\mathrm{d}t
		\end{align*}
		Hence,
		\begin{equation}
			v(t)=\mathbb{E}[\bar{\mathcal{Y}}_t^2],\,t<T.
		\end{equation}
		Thus, we may conclude that the SDE \eqref{eqvariance} has a unique strong solution.
	\end{proof}

	In the next proposition, we provide the explicit expression for the solution of the MV-SDE \eqref{eqvariance}. 
	\begin{proposition}
		The MV-SDE \eqref{eqvariance} has an explicit solution given by
		\begin{equation}
			\mathcal{Y}_t=\dfrac{1}{\sqrt{g(t)}}\dint_0^t\sqrt{g(s)}\,\mathrm{d}W_s, t<T,\label{eqsolutionvariance}
		\end{equation}
		where,
		\begin{equation}
			g(t)= I_1(2 \sqrt{2} \sqrt{T}) K_0(2 \sqrt{2} \sqrt{T - t}) + K_1(2 \sqrt{2} \sqrt{T}) I_0(2 \sqrt{2} \sqrt{T - t}).
		\end{equation}
		Here, $I_n$ is the modified Bessel function of the first kind and $K_n$ is the modified Bessel function of the second kind.
	\end{proposition}
	\begin{proof}
		In view of the equality of  MV-SDE \eqref{eqvariance} and the SDE
		\eqref{eqvariancef}, their unique strong solutions are identical so that we
		can use the explicit expression derived for \eqref{eqvariancef} to find an explicit solution for MV-SDE \eqref{eqvariance}. Employing Lemma \ref{lmsdebeta}, the solution to the linear SDE
		\eqref{eqvariance} is given by the following expression:
		\begin{equation}
			\bar{\mathcal{Y}}_t=\exp\left(- \int_{0}^{t}  \frac{v(u)}{T-u} \, \mathrm{d}u\right)\left[\int_{0}^{t}\exp\left( \int_{0}^{s}
			\frac{v(u)}{T-u} \, \mathrm{d}u\right)  \mathrm{d}W_s\right], \quad t \in [0, T).
		\end{equation}
		Using the structure of the coefficient function $v$, namely
		\begin{equation}
			v(t)=\dfrac{T-t}{2}\,\,\dfrac{g'(t)}{g(t)},\,t<T,
		\end{equation}
		the solution to the linear SDE \eqref{eqvariancef} can be written in the
		concise form
		\begin{align*}
			\bar{\mathcal{Y}}_t&=\exp\left(- \int_{0}^{t}  \frac{g'(u)}{2g(u)} \, \mathrm{d}u\right)\left[\int_{0}^{t}\exp\left( \int_{0}^{s}  \frac{g'(u)}{2g(u)} \, \mathrm{d}u\right)  \mathrm{d}W_s\right], \quad t \in [0, T)\\
			&=\dfrac{1}{\sqrt{g(t)}}\dint_0^t\sqrt{g(s)}\,\mathrm{d}W_s, \quad t \in [0, T).
		\end{align*}
		Recalling that the solution coincides with the one for MV-SDE
		\eqref{eqvariance} completes the proof.
	\end{proof}
	\begin{proposition}
		The process $\mathcal{Y}$ is a centred Gaussian process with covariance function given by:
		\begin{equation}
			\text{Cov}(\mathcal{Y}_s,\mathcal{Y}_t)=\dfrac{T-s}{2}\,\,\dfrac{g'(s)}{\sqrt{g(s)g(t)}}.\label{eqCovariance}
		\end{equation} 
		Moreover, 
		\begin{equation}
			\lim\limits_{t \rightarrow T}\text{Var}(\mathcal{Y}_t)=0.
		\end{equation}
	\end{proposition}
	\begin{proof}
		Due to \eqref{eqsolutionvariance}, it is clear that $\mathcal{Y}$ is a centred Gaussian process. The covariance function is given by:
		\begin{align*}
			&\text{Cov}(\mathcal{Y}_s,\mathcal{Y}_t)=\dfrac{1}{\sqrt{g(s)g(t)}}\dint_0^sg(u)\,\mathrm{d}u\\
			&=\dfrac{1}{\sqrt{g(s)g(t)}}\left[I_1(2\sqrt{2T}) \int_0^s K_0(2\sqrt{2 (T - u)}) \mathrm{d}u+ K_1(2 \sqrt{2T}) \int_0^s I_0(2 \sqrt{2 (T - u)})\mathrm{d}u\right].
		\end{align*} 
		We have
		\begin{equation}
			\int_0^s K_0\left(2 \sqrt{2 (T - u)}\right) \, \mathrm{d}u = \frac{\sqrt{T - s} \, K_1\left(2 \sqrt{2} \sqrt{T - s}\right)}{\sqrt{2}} - \frac{\sqrt{T} \, K_1\left(2 \sqrt{2} \sqrt{T}\right)}{\sqrt{2}}\label{eqIntbesselK}
		\end{equation}
		and
		\begin{equation}
			\int_0^s I_0\left(2 \sqrt{2 (T - u)}\right) \, \mathrm{d}u = -\frac{\sqrt{T - s} \, I_1\left(2 \sqrt{2} \sqrt{T - s}\right)}{\sqrt{2}} + \frac{\sqrt{T} \, I_1\left(2 \sqrt{2} \sqrt{T}\right)}{\sqrt{2}}.\label{eqIntbesselI}
		\end{equation}
		Hence,
		\begin{align*}
			\text{Cov}(y_s,Y_t)&=\sqrt{\dfrac{T-s}{2g(s)g(t)}}\left[I_1(2\sqrt{2T}) K_1\left(2 \sqrt{2} \sqrt{T - s}\right)- K_1(2 \sqrt{2T}) I_1\left(2 \sqrt{2} \sqrt{T - s}\right)\right]\\
			&=\dfrac{T-s}{2}\,\,\dfrac{g'(s)}{\sqrt{g(s)g(t)}}.
		\end{align*} 
		Let us now show that
		\begin{equation}
			\lim\limits_{t \rightarrow T}\text{Var}(\mathcal{Y}_t)=0.
		\end{equation}
		We have
		\begin{equation}
			\text{Var}(\mathcal{Y}_t)=v(t)=\dfrac{\sqrt{T-t}}{\sqrt{2}}\,\,\dfrac{ I_1(2 \sqrt{2} \sqrt{T}) K_1(2 \sqrt{2} \sqrt{T - t}) - K_1(2 \sqrt{2} \sqrt{T}) I_1(2 \sqrt{2} \sqrt{T - t})}{ I_1(2 \sqrt{2} \sqrt{T}) K_0(2 \sqrt{2} \sqrt{T - t}) + K_1(2 \sqrt{2} \sqrt{T}) I_0(2 \sqrt{2} \sqrt{T - t})}.
		\end{equation}
		Using the following series expansions
		\begin{multline}\label{eqseriesexpansionsbesselK_0}
			K_0(z) \propto\left(-\gamma+\frac{1}{4}(1-\gamma) z^2+\frac{1}{128}(3-2 \gamma) z^4+\ldots\right)\\-\log \left(\frac{z}{2}\right)\left(1+\frac{z^2}{4}+\frac{z^4}{64}+\ldots\right) / ;(z \rightarrow 0)
		\end{multline}
		\begin{multline}\label{eqseriesexpansionsbesselK_1}
			K_1(z) \propto \frac{1}{z}+\frac{z}{4}\left(2 \gamma-1+\frac{1}{8}\left(2 \gamma-\frac{5}{2}\right) z^2+\frac{1}{192}\left(2 \gamma-\frac{10}{3}\right) z^4+\ldots\right)\\
			+\frac{z}{2} \log \left(\frac{z}{2}\right)\left(1+\frac{z^2}{8}+\frac{z^4}{192}+\ldots\right) /(z \rightarrow 0)
		\end{multline}
		\begin{equation}\label{eqseriesexpansionsbesselI_nu}
			I_v(z) \propto \frac{1}{\Gamma(v+1)}\left(\frac{z}{2}\right)^v\left(1+\frac{z^2}{4(v+1)}+\frac{z^4}{32(v+1)(v+2)}+\ldots\right) / ;(z \rightarrow 0)
		\end{equation}
		we get the following limits:
		\begin{equation}
			\lim\limits_{t \rightarrow T} \sqrt{T - t} \,K_1(2 \sqrt{2} \sqrt{T - t})=\dfrac{\sqrt{2}}{4},\,\,\, \lim\limits_{t \rightarrow T} \sqrt{T - t} \,I_1(2 \sqrt{2} \sqrt{T - t})=0,\label{eqBessellimits}
		\end{equation}
		\begin{equation}
			\lim\limits_{t \rightarrow T} \,K_0(2 \sqrt{2} \sqrt{T - t})=+\infty,\,\,\, \lim\limits_{t \rightarrow T} \,I_0(2 \sqrt{2} \sqrt{T - t})=1.
		\end{equation}
		Thus,
		\begin{equation*}
			\lim\limits_{t \rightarrow T}\text{Var}(\mathcal{Y}_t)=0.
		\end{equation*}
		This completes the proof.
	\end{proof}
	From the explicit expression of the process $\mathcal{Y}$, we can deduce that for all $t<T$
	\begin{equation}
		\sqrt{g(t)}\,\mathcal{Y}_t=\dint_0^t\sqrt{g(s)}\,\mathrm{d}W_s, \,\,\,t<T,\label{eqgX}
	\end{equation}
	Hence, we have the following Corollary:
	\begin{corollary}\label{corindepincreVariance}
		The process $N=(N_t ,t<T)$ given by
		\begin{equation}
			N_t=\sqrt{g(t)}\,\mathcal{Y}_t,
		\end{equation}
		is with independent increments.
	\end{corollary}
	\begin{proposition}
		The process $\mathcal{Y}$ is a pinned process, i.e.,
		\begin{equation}
			\mathbb{P}\left(\lim\limits_{t \rightarrow T} \mathcal{Y}_t=0\right)=1.
		\end{equation}
	\end{proposition}
	\begin{proof}
		From Corollary \ref{corindepincreVariance} we see that $(N_t, t<T)$ is a continuous martingale with bracket
		\begin{equation*}
			\langle N\rangle_t=\displaystyle\int_{0}^tg(s)\mathrm{d}s=I_1(2\sqrt{2T}) \int_0^t K_0(2\sqrt{2 (T - u)}) \mathrm{d}u+ K_1(2 \sqrt{2T}) \int_0^t I_0(2 \sqrt{2 (T - u)})\mathrm{d}u.
		\end{equation*}
		Using \eqref{eqIntbesselK} and \eqref{eqIntbesselI} we obtain
		\begin{equation*}
			\langle N\rangle_t=\sqrt{\dfrac{T-t}{2}}\left[I_1(2\sqrt{2T}) K_1\left(2 \sqrt{2} \sqrt{T - t}\right)- K_1(2 \sqrt{2T}) I_1\left(2 \sqrt{2} \sqrt{T - t}\right)\right].
		\end{equation*}
		Using again \eqref{eqseriesexpansionsbesselK_1} and \eqref{eqseriesexpansionsbesselI_nu}, we obtain 
	
		\begin{equation*}
			\langle N\rangle_T=\lim\limits_{t \rightarrow T}\langle N\rangle_t=\dfrac{I_1(2\sqrt{2T})}{4}<+\infty.
		\end{equation*}
		Thus, $N$ is a (true) martingale bounded in $L^2$. Furthermore, $N_t$ converges almost surely as $t$ goes to $T$. Consequently, using the fact that $\lim\limits_{t\rightarrow T}\sqrt{g(t)}=+\infty$ and 
		\begin{equation}
			\mathcal{Y}_t=\dfrac{N_t}{\sqrt{g(t)}}
		\end{equation}
		we obtain the desired result.
	\end{proof} 
	\subsection{The Nonlinear Case ($\alpha>0$)}
	\hspace{0,6cm}While the case $\alpha=1$ offers explicit solutions, the dynamics become significantly more intricate when $\alpha>0$. The nonlinear dependence of the drift on the second moment introduces challenges in establishing existence and uniqueness of solutions. In this subsection, we study the general MV-SDE:
	\begin{equation}
		\left\{\begin{array}{l}
			\mathrm{d} Y_t=-(\mathbb{E}[Y_t^2])^\alpha[Y_t/(T-t)] \mathrm{d} t+\mathrm{d} W_t,\,\, t<T\\
			Y_0=x\geq 0,
		\end{array}\right.\label{eqvariancegeneral}
	\end{equation}
	where $\alpha>0$. Before addressing the existence and uniqueness of \eqref{eqvariancegeneral}, we first consider the existence and uniqueness of the following ODE:
	\begin{equation}
		\left\{\begin{array}{l}
			f'(t)=-2[f^{\alpha+1}(t)/(T-t)]+1,\,t<T\\
			f(0)=x\geq 0
		\end{array}\right.\label{eqODEalpha}
	\end{equation}
	\begin{proposition}\label{propODEalpha}
		The ODE \eqref{eqODEalpha} 
		admits a bounded non-negative solution on $[0,T]$.
	\end{proposition}
	\begin{proof} 
	
		The existence of a local solution to the ODE \eqref{eqODEalpha} is evident. To establish that this solution is global, it is sufficient to demonstrate that the limit of $f$ exists as $t$ approaches $T$. First, we observe that the function $f$ cannot reach $0$ at any point before $T$, except at $t = 0$ if $x = 0$. Indeed, suppose $t_0$ is the first time at which $f$ attains $0$, i.e., $f(t_0) = 0$. From the ODE \eqref{eqODEalpha}, we find that the derivative at $t_0$ is $f'(t_0) = 1$, which implies that $f$ is increasing at $t_0$. However, for $f(t_0) = 0$, the function would necessarily have had to decrease prior to $t_0$, contradicting the condition $f'(t_0) > 0$.	Let us now prove that the limit of $f$ exists as $t$ approaches $T$. For all $0 \leq t < T$, we have  
		\begin{equation}
			f'(t) - 1 = -2 \frac{f^{\alpha+1}(t)}{T - t} \leq 0.
		\end{equation}  
		Thus,  
		\begin{equation}
			\big( f(t) + T - t \big)' = f'(t) - 1 \leq 0.
		\end{equation}  
		This shows that the function $t \mapsto f(t) + T - t$ is decreasing and non-negative. Therefore, the limit of $f$ exists as $t$ approaches $T$. Furthermore, again since $t \mapsto f(t) + T - t$ is decreasing, for all $t\in[0,T]$, $0\leq f(t)\leq x+T$.
	\end{proof}
	\begin{proposition}
		The MV-SDE \eqref{eqvariancegeneral} has a unique strong solution.
	\end{proposition}
	\begin{proof}
		We consider the following SDE:
		\begin{equation}
			\left\{\begin{array}{l}
				\mathrm{d} \hat{Y}_t=-(f(t))^{\alpha}[\hat{Y}_t/(T-t)] \mathrm{d} t+\mathrm{d} W_t,\,\, t<T\\
				\hat{Y}_0=0
			\end{array}\right.\label{eqvariancegeneralf}
		\end{equation}			
		where $f$ is the solution to the ODE \eqref{eqODEalpha}. It follows from Lemma \ref{lmsdebeta} that the SDE \eqref{eqvariancef} admits a unique strong solution. Moreover, for all $t<T$, we have
		\begin{equation}
			\mathbb{E}[\hat{Y}_t^2]=\int_{0}^{t}\exp\left( -2\int_{s}^{t}  \frac{(f(u))^{\alpha}}{T-u} \, \mathrm{d}u\right)  \mathrm{d}s.\label{eqsecondmomentofhatY}
		\end{equation}
		Applying Itô's formula, we obtain  
		\begin{equation}
			\hat{Y}_t^2=-2\int_{0}^{t}(f(s))^{\alpha}\dfrac{\hat{Y}^2_s}{T-s}\mathrm{d}s+2\int_{0}^{t}\hat{Y}_s\mathrm{d}W_s+t.\label{eqitoYalpha}
		\end{equation}
		It follows from \eqref{eqsecondmomentofhatY} that the process $(\int_{0}^{t}\hat{Y}_s\mathrm{d}W_s, 0\leq t<T)$ is a true martingale. As a consequence, it follows from \eqref{eqitoYalpha} that  
		\begin{equation}
			\mathrm{d}\mathbb{E}[\hat{Y}_t^2]=-2(f(t))^{\alpha}\dfrac{\mathbb{E}[\hat{Y}^2_t]}{T-t}\mathrm{d}t+\mathrm{d}t.\label{eqEY_t^2'alpha}
		\end{equation}
		Using the fact that $f$ is the solution to the ODE \eqref{eqODEalpha} and \eqref{eqEY_t^2'alpha}, we obtain  
		\begin{align*}
			\mathrm{d}(f(t)-\mathbb{E}[\hat{Y}_t^2])&=-2\dfrac{(f(t))^{\alpha}}{T-t}(f(t)-\mathbb{E}[\hat{Y}_t^2])\mathrm{d}t
		\end{align*}  
		Hence,  
		\begin{equation}
			f(t)=\mathbb{E}[\hat{Y}_t^2],\,t<T.\label{eqfEY_t^2}
		\end{equation} 
		Hence, \eqref{eqvariancegeneralf} and \eqref{eqvariancegeneral} represent the same SDE. Thus, the MV-SDE \eqref{eqvariancegeneral} has a unique strong solution.  
	\end{proof}
	\begin{proposition}\label{propLimitoff}
		We have 
		\begin{equation}
			\lim\limits_{t \rightarrow T}\mathbb{E}[Y^2_t]=0.\label{eqlimitEX_t^2}
		\end{equation}
	\end{proposition}
	\begin{proof}
		For the proof of \eqref{eqlimitEX_t^2}, from \eqref{eqfEY_t^2}, we have
		\begin{equation}
			\lim\limits_{t \rightarrow T}\mathbb{E}[Y^2_t]=\lim\limits_{t \rightarrow T}f(t).
		\end{equation}	
		It is proved in Proposition \ref{propODEalpha} that the limit of $f$ exists as $t$ approaches $T$. We only need to show that $\lim\limits_{t \rightarrow T}f(t)=0$. To do so, let assume that $\lim\limits_{t \rightarrow T}f(t)=L>0$. Hence, for all $\varepsilon>0$, there exists  $\delta>0$ such that if $\vert t-T\vert< \delta$ we have 
		\begin{equation}
			\vert f(t)-L\vert <\varepsilon.
		\end{equation}
		For $\varepsilon=\frac{L}{2}$, there exists $\delta>0$ such that for all $T-\delta<t<T$, we have
		\begin{equation}
			\frac{L}{2}<f(t)<\frac{3L}{2}
		\end{equation}
		Thus, we have
		\begin{align}
			f(t)-f(T-\delta)&=\dint_{T-\delta}^{t}f'(s)\mathrm{d}s\\
			&=-2\dint_{T-\delta}^{t}\dfrac{f^{\alpha+1}(s)}{T-s}\mathrm{d}s+t-T+\delta.
		\end{align}
		Hence,
		\begin{equation*}
			2\left(\dfrac{3L}{2}\right)^{\alpha+1}\log\left(\dfrac{T-t}{\delta}\right)+t-T+\delta<f(t)-f(T-\delta)<2\left(\dfrac{L}{2}\right)^{\alpha+1}\log\left(\dfrac{T-t}{\delta}\right)+t-T+\delta.
		\end{equation*}
		This implies that $\lim\limits_{t \rightarrow T}f(t)=-\infty$. This contradicts the fact that  $\lim\limits_{t \rightarrow T}f(t)=L>0$. Thus, $L=0$.
	\end{proof}
In the next proposition, we investigate the pinned property of the process $Y$. However, since $f$, the solution to the ODE \eqref{eqODEalpha}, is not explicitly available, we cannot directly prove the pinned property as we did in the case $\alpha=1$. To address this difficulty, we rely on a result from \cite{HR}, which is presented in the Appendix, where the pinned property for processes of a similar form is studied.
	\begin{proposition}
		The solution to the MV-SDE \eqref{eqvariancegeneral} is a pinned process.
	\end{proposition}
	\begin{proof}
		We aim to demonstrate that
		\begin{equation*}
			\mathbb{P}\left(\lim\limits_{t \rightarrow T} Y_t=0\right)=1.
		\end{equation*}	
		To achieve this, due to Lemma \ref{lmpinnedproperty}, it suffices to show that
		\begin{equation*}
			\lim_{t \rightarrow T} \displaystyle\int_0^t  \dfrac{(f(s))^{\alpha}}{T-s} \mathrm{d}s= +\infty.
		\end{equation*}
		where $f$ is the solution to the ODE \eqref{eqODEalpha}. Let $0<\varepsilon<T$, for any $t\in (\varepsilon, T)$, it follows from \eqref{eqODEalpha} that
		\begin{align*}
			\displaystyle\int_0^t  \dfrac{(f(s))^{\alpha}}{T-s} \mathrm{d}s&\geq \displaystyle\int_{\varepsilon}^t  \dfrac{(f(s))^{\alpha}}{T-s} \mathrm{d}s=\displaystyle\int_{\varepsilon}^t  \dfrac{1-f'(s)}{2f(s)} \mathrm{d}s\\
			&\geq -\dfrac{1}{2}\displaystyle\int_{\varepsilon}^t  \dfrac{f'(s)}{f(s)} \mathrm{d}s\\
			&=-\dfrac{1}{2}\log\left(\dfrac{f(t)}{f(\varepsilon)}\right).
		\end{align*}
		Thus, applying Proposition \ref{propLimitoff}, we obtain the desired result.
	\end{proof}
	\section{McKean–Vlasov SDEs with Expectation-Dependent Coefficients}
	\hspace{0,6cm}In this section, we generalize the results of the previous sections by considering McKean–Vlasov stochastic differential equations (SDEs) where the drift and the volatility depend on the time and the expectation of a general function of the process. Specifically, we study SDEs of the form: 
	\begin{equation}
		\left\{\begin{array}{l}
			\mathrm{d} \xi_t=-\mu(t,\mathbb{E}[\varphi_1(\xi_t)])[\xi_t/(T-t)] \mathrm{d} t+\sigma(t,\mathbb{E}[\varphi_2(\xi_t)]) \mathrm{d} W_t,, \quad t <T \\
			\xi_0=0
		\end{array}\right.\label{eqMV-SDEgeneral}
	\end{equation}
	where $\mu$ and $\sigma$ are deterministic functions that may depend on time $t$ and the expectation of given functions $\varphi_1$ and $\varphi_2$ of the process
	
	McKean–Vlasov equations are fundamental in describing systems where the dynamics of individual components depend on the collective behavior of the system, often represented by aggregate statistics such as the mean or variance of the population. In previous sections, we focused on specific cases. Section 2 dealt with SDEs where the drift term explicitly depended on the expectation raised to the power $\alpha$. This captures systems where the interaction strength is governed by power-law dependencies of the mean. Section 3 extended this to systems where the drift term depended on the second moment raised to the power $\alpha$, providing insights into models with higher-order moment dependencies. While these cases provide valuable insights into the behavior of systems with power-weighted expectation terms, they are limited to specific functional dependencies in the drift term and assume a constant diffusion coefficient.

In this section, we generalize the framework to consider both drift and diffusion coefficients as functions of the time $t$ and the expectation of general functions $\varphi_1$ and $\varphi_2$ of the process. This generalization allows us to model interactions driven by non-linear transformations of the process, beyond power laws. Beyond developing mathematics introducing additionally time and mean-field in all coefficient functions, i.e. $\mu(t, .)$ and $\sigma(t,.)$, reveals greater flexibility in modelling evolving systems influenced by external or temporal factors. The main result of this section establishes the existence and uniqueness of solutions to the above SDE under appropriate conditions on $\varphi$. Additionally, we prove that the second moment of the process satisfies:
	\begin{equation}
		\mathbb{E}[\xi_t^2]\mapsto 0 \text{   as  } t\mapsto T.
	\end{equation}
	This result highlights the dissipative nature of the drift term, which drives the process toward stabilization over time. Moreover, we show that the solution to the MV-SDE \eqref{eqMV-SDEgeneral} satisfies the pinned property.
	\begin{hy}
		Suppose that:
		\begin{enumerate}
			\item [(i)] the function $\mu:[0,T]\times\mathbb{R}:\longrightarrow \mathbb{R}$ is continuous and positive.
			\item [(ii)] the function $\sigma:[0,T]\times\mathbb{R}:\longrightarrow \mathbb{R}$ is continuous on $[0,T)\times\mathbb{R}$ and does not vanish, furthermore, $$\sigma^2(t,x)\leq h(t) \text{ and } \int_0^Th(t)\mathrm{d}t<+\infty$$
			\item [(iii)] the functions $\varphi_1$ and $\varphi_2$ are measurable functions such that:

			\begin{center}
				the functions $t\mapsto \mathbb{E}[\varphi_j(W_t)]$, $j=1, 2$, are continuous on $\mathbb{R}_+$.	
			\end{center}
		\end{enumerate}
	\end{hy}
	The condition (iii) is natural but implicit. However, to ensure this condition holds, various explicit conditions on \( \varphi_j \), \( j = 1, 2 \), can be found in the literature. For example: \begin{enumerate} \item If \( \varphi_j \), \( j = 1, 2 \), are bounded and continuous, then the continuity of the functions \( t \mapsto \mathbb{E}[\varphi_j(W_t)] \), \( j = 1, 2 \), follows immediately from the weak convergence of probability measures. \item If \( \varphi_j \), \( j = 1, 2 \), are bounded, integrable on \( \mathbb{R} \), and continuous at \( 0 \), then the continuity of the functions \( t \mapsto \mathbb{E}[\varphi_j(W_t)] \), \( j = 1, 2 \), follows immediately from Lebesgue's dominated convergence theorem. \item If \( \varphi_j \), \( j = 1, 2 \), satisfy a growth condition (e.g., polynomial growth) and are continuous at \( 0 \), then the continuity of the functions \( t \mapsto \mathbb{E}[\varphi_j(W_t)] \), \( j = 1, 2 \), also follows from Lebesgue's dominated convergence theorem. \end{enumerate}
	
	Now we are in position to investigate the existence and uniqueness of the MV-SDE \eqref{eqMV-SDEgeneral}. But before, we first investigate the existence of the solution of the following ODE:
	\begin{equation}
		\left\{\begin{array}{l}
			\eta'(t)=-2\mu(t,\phi_1(\eta(t)))[\eta(t)/(T-t)]+\sigma^2(t,\phi_2(\eta(t))),\,t<T\\
			\eta(0)=x\geq 0
		\end{array}\right.\label{eqODEgeneral}
	\end{equation}
	where,
	\begin{equation}
		\phi_j(r)=\mathbb{E}[\varphi_j(W_r)]=\displaystyle\int_{\mathbb{R}}\varphi_j(y)\dfrac{1}{\sqrt{2\pi r}}\exp\left(-\dfrac{y^2}{2r}\right)\mathrm{d}y,\,\,r>0,\,\,j=1,\,2.
	\end{equation}
	\begin{proposition}\label{propODEgeneral}
		The ODE \eqref{eqODEgeneral} admits a solution on $[0,T]$.
	\end{proposition}
	\begin{proof}
		Since the functions $r \mapsto \phi_j(r)$, $j=1, 2$, are continuous on $\mathbb{R}_+$, the existence of a local solution to the ODE \eqref{eqODEgeneral} is straightforward. To prove that the solution is global, it suffices to show that $\eta$ is non-negative and that the limit of $\eta$ exists as $t \to T$. First, we observe that the function $\eta$ cannot reach $0$ at any point before $T$, except at $t = 0$ if $x = 0$. Indeed, suppose $t_0$ is the first time at which $\eta$ attains $0$, i.e., $\eta(t_0) = 0$. From the ODE \eqref{eqODEgeneral}, we find that the derivative at $t_0$ is 
		$$\eta'(t_0) = \sigma^2(t_0,\varphi_2(0))>0,$$ 
		which implies that $\eta$ is increasing at $t_0$. However, for $\eta(t_0) = 0$, the function would necessarily have had to decrease prior to $t_0$, contradicting the condition $\eta'(t_0) > 0$. Let us show that the limit of $\eta$ exists as $t \to T$, note that $\mu$ is non-negative, which implies that for all $0\leq t<T$
		\begin{equation}
			\eta'(t)-\sigma^2(t,\phi_2(\eta(t)))\leq 0,
		\end{equation}
		and hence, the function 
		\begin{equation}
			t\mapsto \eta(t)+\int_t^T\sigma^2(s,\phi_2(\eta(s)))\mathrm{d}s,
		\end{equation}
		is non-decreasing on $[0,T)$. Since it is also non-negative, it has a limit as $t$ goes to $T$. Therefore, $\eta$ has a finite limit as $t \to T$, completing the proof.
	\end{proof}
	
	\begin{proposition}
		The MV-SDE \eqref{eqMV-SDEgeneral} has a unique strong solution.
	\end{proposition}
	\begin{proof}
		Let $\eta$ be a solution to the ODE \eqref{eqODEgeneral}. We consider the following SDE:
	\begin{equation}
        \left\{\begin{array}{l}
	\mathrm{d} \bar{\xi}_t=-\mu(t,\phi_1(\eta(t))) \dfrac{\bar{\xi}_t}{T-t} 
        \mathrm{d} t+\sigma(t,\phi_2(\eta(t))) \mathrm{d} W_t,\,t<T\\
	\bar{\xi}_0= 0
	\end{array}\right.\label{eqMV-SDEgeneralgeneral}
        \end{equation}
		It follows from Lemma \ref{lmsdebeta} that the SDE \eqref{eqMV-SDEgeneralgeneral} admits a unique strong solution. Moreover, the solution can be explicitly expressed in the following form: 
		\begin{equation}
			\bar{\xi}_t=G(t) \int_0^t \frac{\sigma(s,\phi_2(\eta(s)))}{G(s)} \mathrm{d} W_s, \quad t \in[0,T),\label{eq11}
		\end{equation}
		where
		\begin{equation}
			G(t):=\exp \left(-\int_0^t \dfrac{\mu(s,\phi_1(\eta(s)))}{T-s} \mathrm{d} s\right)
		\end{equation}
		is a function of class $C^1$ on $[0,T)$. Since the integrand in the stochastic integral in \eqref{eq11} is deterministic and locally bounded, $\bar{\xi}$ is a Gaussian process with mean zero and second moment given by
		\begin{equation}
			\mathbb{E}[\bar{\xi}_t^2] =G^2(t)\int_0^t \frac{\sigma^2(s,\phi_2(\eta(s)))}{G^2(s)} \mathrm{d} s, \quad 0 \leq  t<T.
		\end{equation}
		Hence, the function $t\mapsto \mathbb{E}[\bar{\xi}_t^2]$ is of class $C^1$ on $[0,T)$. Therefore, the process $$\left(\int_{0}^{t}\bar{\xi}_s\sigma(s,\phi_2(\eta(s)))\mathrm{d}W_s, 0\leq t<T\right)$$ is a true martingale. Applying Itô's formula, we obtain  
		\begin{equation}
			\bar{\xi}_t^2=-2\int_{0}^{t}\mu(s,\phi_1(\eta(s)))\dfrac{\bar{\xi}^2_s}{T-s}\mathrm{d}s+2\int_{0}^{t}\bar{\xi}_s\sigma(s,\phi_2(\eta(s)))\mathrm{d}W_s+\int_{0}^{t}\sigma^2(s,\phi_2(\eta(s)))\mathrm{d}s.\label{eqitoYmusigma}
		\end{equation}
		As a consequence, it follows from \eqref{eqitoYmusigma} that  
		\begin{equation}
			\mathrm{d}\mathbb{E}[\bar{\xi}_t^2]=-2\mu(t,\phi_1(\eta(t)))\dfrac{\mathbb{E}[\bar{\xi}^2_t]}{T-t}\mathrm{d}t+\sigma^2(t,\phi_2(\eta(t)))\mathrm{d}t.\label{eqEY_t^2'musigma}
		\end{equation}
		Using the fact that $f$ is the solution to the ODE \eqref{eqODEgeneral} and \eqref{eqEY_t^2'musigma}, we obtain  
		\begin{align*}
			\mathrm{d}(\eta(t)-\mathbb{E}[\bar{\xi}_t^2])&=-2\dfrac{\mu(t,\phi_1(\eta(t)))}{T-t}(\eta(t)-\mathbb{E}[\bar{\xi}_t^2])\mathrm{d}t
		\end{align*}  
		Hence,  
		\begin{equation*}
			\eta(t)=\mathbb{E}[\bar{\xi}_t^2],\,t<T.
		\end{equation*} 
		Consequently, for all $t<T$
		\begin{equation}
			\phi_j(\eta(t))=\phi\left(\mathbb{E}[\bar{\xi}_t^2]\right)=\displaystyle\int_{\mathbb{R}}\varphi_j(y)\dfrac{1}{\sqrt{2\pi \mathbb{E}[\bar{\xi}_t^2]}}\exp\left(-\dfrac{y^2}{2\mathbb{E}[\bar{\xi}_t^2]}\right)\mathrm{d}y,\,\,j=1,\,2.
		\end{equation}
		Since the solution to \eqref{eqMV-SDEgeneralgeneral} is a centred Gaussian process, we obtain
		\begin{equation*}
			\phi_j(\eta(t))=\mathbb{E}[\varphi_j(\bar{\xi}_t)],\,\,j=1,\,2.
		\end{equation*}
		Consequently, \eqref{eqMV-SDEgeneralgeneral} and \eqref{eqMV-SDEgeneral} corespond to the same SDE. Thus, the MV-SDE \eqref{eqMV-SDEgeneral} admits a unique strong solution.  
	\end{proof}
	To understand the asymptotic behavior of the process as it approaches the boundary of its domain, we now demonstrate that the second moment of $\xi_t$ converges to zero as $t \to T$.
	\begin{proposition} 
		Assume in addition that, for all $t\geq 0$, the function $x\longrightarrow \mu(t,x)$ is monotone. We have 
		\begin{equation}
			\lim\limits_{t \rightarrow T}\mathbb{E}[\xi^2_t]=0.\label{eqlimitgeneralEX_t^2}
		\end{equation}	
	\end{proposition}
	\begin{proof}
		We have shown in Proposition \ref{propODEgeneral} that $\lim\limits_{t \rightarrow T}\eta(t)$ exists. Assume, for contradiction, that $\lim\limits_{t \rightarrow T}\eta(t)=L>0$. Hence, there exists $\delta>0$ such that for all $T-\delta<t<T$, we have
		\begin{equation}
			\frac{L}{2}<\eta(t)<\frac{3L}{2}
		\end{equation}
		Thus, we have
		\begin{align}
			\eta(t)-\eta(T-\delta)&=\dint_{T-\delta}^{t}\eta'(s)\mathrm{d}s\nonumber\\
			&=-2\dint_{T-\delta}^{t}\mu(s,\phi_1(\eta(s)))\dfrac{\eta(s)}{T-s}\mathrm{d}s+\dint_{T-\delta}^{t}\sigma^2(s,\phi_2(\eta(s)))\mathrm{d}s,
		\end{align}
		where, we recall,
		$$ \phi_j(\eta(s))=\displaystyle\int_{\mathbb{R}}\varphi_j(y)\frac{1}{\sqrt{2\pi \eta(s)}}\exp\left(-\dfrac{y^2}{2 \eta(s)}\right)\mathrm{d}y,\,\,j=1,\,2.$$
		Hence, for all $T-\delta<s<t$, we have
		\begin{equation}
			\dfrac{1}{\sqrt{3}}\, \mathbb{E}[\varphi_j(B_{\frac{L}{2}})]	\leq \phi_j(\eta(s)) \leq \sqrt{3}\, \mathbb{E}[\varphi_j(B_{\frac{3L}{2}})],\,\,j=1,\,2.
		\end{equation}
		Thus,
		\begin{equation}
			\inf\limits_{s\in[0,T]}\mu\left(s,\dfrac{1}{\sqrt{3}}\, \mathbb{E}\left[\varphi_1(B_{\frac{L}{2}})\right]\right)	\leq \mu(s,\phi_1(\eta(s))) \leq \sup\limits_{s\in[0,T]}\mu\left(s,\sqrt{3}\, \mathbb{E}\left[\varphi_1\left(B_{\frac{3L}{2}}\right)\right]\right),
		\end{equation}
		Which implies that 
		\begin{small}
			\begin{multline*}
				-\dfrac{L}{2}	\inf\limits_{s\in[0,T]}\mu\left(s,\dfrac{1}{\sqrt{3}}\, \mathbb{E}\left[\varphi_1(B_{\frac{L}{2}})\right]\right)\log\left(\dfrac{T-t}{\delta}\right)	\leq \dint_{T-\delta}^{t}\mu(s,\phi_1(\eta(s)))\dfrac{\eta(s)}{T-s}\mathrm{d}s \\\leq- \dfrac{3L}{2} \sup\limits_{s\in[0,T]}\mu\left(s,\sqrt{3}\, \mathbb{E}\left[\varphi_1\left(B_{\frac{3L}{2}}\right)\right]\right)\log\left(\dfrac{T-t}{\delta}\right).
			\end{multline*}
		\end{small}
		Consequently, 
		\begin{equation*}
			\bar{H}(t)\leq \eta(t)-\eta(T-\delta)\leq\hat{H}(t)
		\end{equation*}
		where,
		\begin{equation}
			\bar{H}(t)=\dint_{T-\delta}^{t}\sigma^2(s,\phi_2(\eta(s)))\mathrm{d}s+3L \sup\limits_{s\in[0,T]}\mu\left(s,\sqrt{3}\, \mathbb{E}\left[\varphi_1\left(B_{\frac{3L}{2}}\right)\right]\right)\log\left(\dfrac{T-t}{\delta}\right)
		\end{equation}
		and
		\begin{equation*}
			\hat{H}(t)=\dint_{T-\delta}^{t}\sigma^2(s,\phi_2(\eta(s)))\mathrm{d}s+L\inf\limits_{s\in[0,T]}\mu\left(s,\dfrac{1}{\sqrt{3}}\, \mathbb{E}\left[\varphi_1(B_{\frac{L}{2}})\right]\right)\log\left(\dfrac{T-t}{\delta}\right)
		\end{equation*}
		This implies that $\lim\limits_{t \rightarrow T}\eta(t)=-\infty$. This contradicts the fact that  $\lim\limits_{t \rightarrow T}\eta(t)=L>0$. Thus, $\lim\limits_{t \rightarrow T}\mathbb{E}[\xi^2_t]=0$.
	\end{proof}
	\begin{proposition}
		The solution to the MV-SDE \eqref{eqMV-SDEgeneral} is a pinned process, i.e., 
		\begin{equation*}
			\mathbb{P}\left(\lim\limits_{t \rightarrow T} \xi_t=0\right)=1.
		\end{equation*}
    \end{proposition}	
		\begin{proof}
			We need to show that $\mathbb{P}\left(\lim\limits_{t \rightarrow T} \xi_t=0\right)=1$. Let $\varepsilon>0$. For any $t\in (\varepsilon,T)$, we have
			\begin{align*}
				\displaystyle\int_0^t  \dfrac{\mu(s,\phi_1(\eta(s)))}{T-s} \mathrm{d}s&\geq \displaystyle\int_{\varepsilon}^t  \dfrac{\mu(s,\phi_1(\eta(s)))}{T-s} \mathrm{d}s\\
				&=\displaystyle\int_{\varepsilon}^t  \dfrac{\sigma^2(s,\phi_2(\eta(s)))-\eta'(s)}{2\eta(s)} \mathrm{d}s\\
				&\geq -\dfrac{1}{2}\displaystyle\int_{\varepsilon}^t  \dfrac{\eta'(s)}{\eta(s)} \mathrm{d}s\\
				&\geq -\dfrac{1}{2}\log\left(\dfrac{\mathbb{E}[\xi^2_t]}{\mathbb{E}[\xi^2_{\varepsilon}]}\right).
			\end{align*}
			Thus, 
			\begin{equation*}
				\lim_{t \rightarrow T} \displaystyle\int_0^t  \dfrac{\mu(s,\phi_1(f(s)))}{T-s} \mathrm{d}s= +\infty.
			\end{equation*}
			On the other hand, we have 
			\begin{equation}
				\lim\limits_{t \rightarrow T}\displaystyle\int_0^t \sigma^2(s,\mathbb{E}[\varphi_2(\xi_s)])\mathrm{d}s\leq \lim\limits_{t \rightarrow T}\displaystyle\int_0^t h(s)\mathrm{d}s<+\infty.
			\end{equation}
			The desired result follows from \cite[Proposition 1]{HR}.
		\end{proof}
		\begin{remark}
			The global positivity of the function $\mu$ is not essential. For example, if $\varphi_1$ and $\varphi_2$ are positive functions, it suffices to assume that $\mu$ is positive on $(0,T]\times (0,+\infty)$. This is because when $\varphi_1$ and $\varphi_2$ are positive, the expectations $\mathbb{E}[\varphi_1(W_t)]$ and $\mathbb{E}[\varphi_1(W_t)]$ remain positive for $t>0$.
		\end{remark}

	\appendix
	\section{Supporting Lemmas for Main Results}
	In this section, we present the lemmas used to prove the main results of this paper.
	\begin{lemma}\label{lmsdebeta}
		Let $A:[0,T)\longrightarrow \mathbb{R}$ and $B:[0,T)\longrightarrow \mathbb{R}$ be deterministic, measurable, and locally bounded
		functions. Then, the SDE
		\begin{equation}
			\left\{\begin{array}{l}
				\mathrm{d} V_t=-A(t)V_t \mathrm{d} t+B(t)\mathrm{d} W_t, \quad t \in[0,T) \\
				V_0=x
			\end{array}\right.\label{eqSDEHR}
		\end{equation}
		admits a unique strong solution. Moreover, the solution can be explicitly expressed in the following form: 
		\begin{equation}
			V_t=\exp\left(- \int_{0}^{t}  A(u) \, \mathrm{d}u\right)\left[x+\int_{0}^{t}\exp\left( \int_{0}^{s}  A(u) \, \mathrm{d}u\right) B(s)\,  \mathrm{d}W_s\right], \quad t \in [0, T).
		\end{equation}
	\end{lemma}
	\begin{proof}
		see, e.g. Karatzas and Shreve \cite[Section 5.6]{KS}
	\end{proof}
	\begin{lemma}\label{ODE_variance}
		The ODE \eqref{eqf't2} possesses the explicit unique solution of the form
		
		\begin{equation}
			v(t) = \dfrac{\sqrt{T-t}}{\sqrt{2}}\,\,\dfrac{ I_1(2 \sqrt{2} \sqrt{T}) K_1(2 \sqrt{2} \sqrt{T - t})
				- K_1(2 \sqrt{2} \sqrt{T}) I_1(2 \sqrt{2} \sqrt{T - t})}{ I_1(2 \sqrt{2} \sqrt{T}) K_0(2 \sqrt{2} \sqrt{T - t}) + K_1(2 \sqrt{2} \sqrt{T}) I_0(2 \sqrt{2} \sqrt{T - t})}
			\label{sol_ODE_variance}
		\end{equation}
	\end{lemma}
	\begin{proof}
		To achieve this, we set:
		\begin{equation}
			v(t)=\dfrac{T-t}{2}r(t), t<T.
			\label{eqTransformation}
		\end{equation}
		Substituting \eqref{eqTransformation} into \eqref{eqf't2}, we obtain the
		following ODE:
		\begin{equation*}
			r'(t)+r^2(t)-\dfrac{1}{T-t}r(t)=\dfrac{2}{T-t},\,\,r(0)=0,
		\end{equation*}
		which takes the form of a general Riccati equation. Inspired by
		\cite[page1133]{GR}, we apply the following transformation:
		\begin{equation*}
			r(t)=\dfrac{u'(t)}{u(t)},\,\,t<T.
		\end{equation*}
		Thus, the function $u$ satisfies the ODE:
		\begin{equation}
			u''(t)-\dfrac{1}{T-t}u'(t)-\dfrac{2}{T-t}u(t)=0,\,\,u'(0)=0.
			\label{eqODEBessel}
		\end{equation}
		By setting $s=2 \sqrt{2} \sqrt{T - t}$, we obtain:
		\begin{equation*}
			\dfrac{\mathrm{d}u}{\mathrm{d}t}=-\dfrac{4}{s}\frac{\mathrm{d}u}{\mathrm{d}s}
		\end{equation*}
		and
		\begin{equation*}
			\dfrac{\mathrm{d}^2u}{\mathrm{d}t^2}=\dfrac{16}{s^2}\dfrac{\mathrm{d}^2u}{\mathrm{d}s^2}
			-\dfrac{16}{s^3}\frac{\mathrm{d}u}{\mathrm{d}s}
		\end{equation*}
		Therefore, \eqref{eqODEBessel} can be rewritten as:
		\begin{equation}
			\dfrac{\mathrm{d}^2u}{\mathrm{d}s^2}+\dfrac{1}{s}\dfrac{\mathrm{d}u}{\mathrm{d}s}-u=0.
			\label{eqODEbessel1}
		\end{equation}
		Using \cite[page 972]{GR}, we conclude that the solution to the ODE
		\eqref{eqODEbessel1} is given by
		\begin{equation*}
			C_1\,K_0(s)+C_2\,I_0(s)
		\end{equation*}
		Transforming back, the solution to the ODE \eqref{eqODEBessel} is given by
		\begin{equation*}
			C_1\, K_0(2 \sqrt{2} \sqrt{T - t}) + C_2\, I_0(2 \sqrt{2} \sqrt{T - t}).
		\end{equation*}
		The initial condition leads to
		\begin{equation*}
			C_1=I_1(2 \sqrt{2} \sqrt{T}) \text{ and } C_2=K_1(2 \sqrt{2} \sqrt{T}).
		\end{equation*}
		Therefore, the solution to the ODE \eqref{eqf't2} is of the form
		\eqref{sol_ODE_variance} given in the proposition of this lemma. This concludes the proof.
	\end{proof}
	\begin{lemma}\label{lmpinnedproperty}
		We consider the following SDE:
		\begin{equation}
			\left\{\begin{array}{l}
				\mathrm{d} R_t=h(t)\left(y-R_t\right) \mathrm{d} t+\mathfrak{s}(t) \mathrm{d} W_t, \quad t \in[0,T) \\
				R_0=x
			\end{array}\right.\label{eqSDEHR}
		\end{equation}
		for any $x, y \in \mathbb{R}$, where $h$ and $\mathfrak{s}$ are two continuous functions defined on $[0, T)$ and that $\mathfrak{s}$ is positive, and $W$ is a standard Brownian motion. If the following holds:
		\begin{enumerate}
			\item[(i)] The function $h$ is bounded from below, and $\lim\limits_{t \rightarrow T} \displaystyle\int_0^t h(r) \, \mathrm{d}r = +\infty$
			\item[(ii)] $\lim\limits_{t \rightarrow T} \displaystyle\int_0^t \mathfrak{s}^2(r) \, \mathrm{d}r < +\infty$
		\end{enumerate}
		then, the solution $R$ to the ODE \eqref{eqSDEHR} possesses the pinned property, that is,
		\begin{equation}
			\mathbb{P}(\lim\limits_{t \rightarrow T} R_t=y)=1.
		\end{equation}
	\end{lemma}
	\begin{proof}
		See	\cite[Proposition 1]{HR}
	\end{proof}
\end{document}